\documentclass[12pt]{amsart}

\usepackage{amsthm,amsmath,amssymb}
\usepackage{xypic}
\usepackage{amssymb,amsmath}
\usepackage{latexsym}
\usepackage{lscape}
\usepackage{rotating}
\usepackage[all]{xy}

\numberwithin{equation}{section}

\def \ni {\noindent}

\def\Z{Z\!\!\!Z}

\def\P{I\!\!P}

\def\F{{\mathbb F}}
\def\A{{\mathcal A}}

\def\H{{\mathcal H}}

\def\I{{\mathcal I}}

\def\C{{\mathbb C}}

\def\?{{\bf ??}}

\def\p{\partial}

\def\t{\theta}

\def\tt#1#2{{\t\left[\begin{matrix}{#1}\\ {#2}\end{matrix}\right]}}
\newcommand{\be}{\begin{equation}}
\newcommand{\ee}{\end{equation}}
\newcommand{\bes}{\begin{equation*}}
\newcommand{\ees}{\end{equation*}}
\newtheorem{theorem}{Theorem}[section]
\newtheorem{lemma}[theorem]{Lemma}
\newtheorem{prop}[theorem]{Proposition}
\newtheorem{definition}[theorem]{Definition}
\newtheorem{corollary}[theorem]{Corollary}

\newtheorem{remark}[theorem]{Remark}

\begin{document}

\title{ Plane quartics: the   matrix of  bitangents}
\author{Francesco Dalla Piazza}
\address{Universit\`a ``La Sapienza'', Dipartimento di Matematica, Piazzale A. Moro 2, I-00185, Roma,  Italy}
\email{dallapiazza@mat.uniroma1.it}
\author{Alessio Fiorentino}
\email{fiorentino@mat.uniroma1.it}
\author{Riccardo Salvati Manni}
\email{salvati@mat.uniroma1.it}

\begin{abstract}

 Aronhold's classical result states that a plane quartic can be recovered by the configuration of any Aronhold systems of  bitangents, i.e.    special  $7$-tuples of bitangents such that the six points at which any subtriple of bitangents touches the quartic do not lie on the same conic in the projective plane. Lehavi (cf. \cite{lh}) proved that a smooth plane quartic  can be explicitly reconstructed from its $28$ bitangents; this result improved Aronhold's method of recovering the curve. In a 2011 paper \cite{PSV} Plaumann, Sturmfels and Vinzant introduced an  eight by eight 
 symmetric matrix parametrizing the  bitangents of a  nonsingular plane quartic. The starting  point of their construction is 
 Hesse's result  for which  every smooth quartic curve has exactly $36$ equivalence  classes of
linear symmetric determinantal representations.  
In this paper we tackle the inverse problem, i.e. the construction of the bitangent matrix starting from the 28 bitangents of the plane quartic.  

\end{abstract}
\maketitle

\section{Introduction} 
It is classically known that the number of the bitangents to a non singular curve of degree $d$ in the projective plane is given by the formula $\frac{1}{2}d(d-2)(d^2-9)$. The main properties of the bitangents have been deeply investigated by geometers since the late nineteenth century, particularly with reference to the first non trivial case, namely the case of degree $4$. Aronhold's classical result states that a plane quartic can be recovered by the configuration of any of the $288$ $7$-tuples of bitangents such that the six points at which any subtriple of bitangents touches the quartic do not lie on the same conic in the projective plane; these $7$-tuples of bitangents are known as Aronhold systems. Caporaso and Sernesi proved in \cite{CS} that the general plane quartic is uniquely determined by its $28$ bitangents; furthermore, they extended this result to general canonical curves of genus $g \geq 4$ (cf. \cite{CS2}).  In \cite{lh} Lehavi proved that a non singular plane quartic can be reconstructed from its $28$ bitangents,  providing a method to derive an explicit formula for the curve. These results have improved Aronhold's method of recovering the curve, because the knowledge of both the bitangents and their contact points on the curve is needed to get the configuration of the Aronhold systems, whereas the sole configuration of the bitangents is enough to describe the geometry of the plane quartic. In a 2011 paper \cite{PSV} Plaumann, Sturmfels and Vinzant introduced
an  eight by eight 
 symmetric matrix parametrizing the  bitangents of a  nonsingular plane quartic. The starting  point of their construction is 
 Hesse's result  for which  every smooth quartic curve has exactly $36$ equivalence  classes of
linear symmetric determinantal representations.  
Each determinantal representation corresponds to three quadrics in $\P^3$ intersecting in eight points. Once such a representation is chosen, the quartic can be described as the curve of the degenerate quadrics of the net generated by the three quadrics corresponding to the representation. Since each line of the net that joins two of the eight intersection points is a bitangent of the curve, this leads to define the eight by eight bitangent matrix.\\
\\
\ni We briefly recall this construction. We refer to \cite{Do}  for details and historical notes.\\
Let $\P^3=\P(W)$ where $W$ is a $4$-dimensional vector space, and let
$$
X=\{x_1,\dots,x_8\} = Q_1 \cap Q_2 \cap Q_3 \subset \P^3
$$
 be an unordered set of $8$ distinct points in $\P^3$, which are complete intersection of
three quadrics. If the net
$$
 \Lambda_X:=|H^0(\P^3,\I_X(2))| = \langle Q_1, Q_2 , Q_3\rangle
$$
 contains no quadric of rank $\le 2$, then
  $X$ is called a \emph{regular Cayley octad} and $\Lambda_X$ is called a \emph{regular net}.
Now, let $C_X=C\subset \Lambda_X$ be the curve of the degenerate quadrics of the regular net $\Lambda_X$. If we denote by $\Delta $ the quartic hypersurface of singular quadrics
in the linear system
$|H^0(\P^3,\mathcal{O}(2))|\cong \P^9$
of all the quadrics
of $\P^3$, then $C= \Lambda_X \cap \Delta$, thus $C$ is a plane
quartic, which  is called the {\it Hesse curve} of the net. After
choosing a basis $\{Q_1,Q_2,Q_3\}$ of $H^0(\P^3,\mathcal{I}_X(2))$,
the net can be identified to $\P^2$ by means of the homogeneous coordinates
$z=(z_1,z_2, z_3)$ as follows
$$
\xymatrix{ (z_1,z_2, z_3)\ar@{<->}[r]&
Q(z)=z_1Q_1 + z_2Q_2 + z_3Q_3}
$$
and the curve $C$ has equation det$(A(z))=0$, where
$A(z)$ is the symmetric $4\times 4$ matrix of the bilinear form associated with the generic quadric $Q(z)$. Note that the Hesse curve of a regular net of quadrics is
nonsingular and det$(A(z))=0$ is one of its $36$ determinantal representations \cite{hesse}.\\
Denoting by $L_{ij}=\langle x_i,x_j\rangle$ for each $1\le i < j \le 8$ the line joining
the points $x_i$ and $x_j$ of the Cayley octad, we obtain a set of $28$ lines in $\P^3$. These 28 lines $L_{ij}$ in $\P^3$ are in correspondence to the 28 bitangents of $C$ in $\P^2$, as any of the equations:
\be
\label{bitangent}
x_i^tA(z)x_j = 0,
\ee
actually define one of the bitangents to the curve (cf. Proposition \ref{P:sturm1}).\\
As for the Aronhold systems of $C$, a notable method to reconstruct one of them from the net $\Lambda_X$ is provided by the Steiner embedding (see \cite{Do2} for details). The singular points of the quadrics in $\Lambda_X$ describe a curve $\Gamma$ in $\P^3$, known as the {\it Steiner curve} of the net. As the net is regular, this curve turns out to be a smooth curve of degree $6$. More precisely, there exists an even theta characteristic $\theta$ such that the map $f:C\to\P^3$ sending a point $p \in C$ to the singular point of the corresponding quadric $Q(p)$ in $\P^3$, is defined by the complete linear series $|\omega_c + \theta|$. In particular, the map $f$ is an isomorphism on the image $f(C)=\Gamma$, and a bijection is defined between classes of regular nets of quadrics in $\P^3$ up to projective equivalences and isomorphism classes of smooth curves of genus $3$ associated with a fixed even theta characteristic. 
If an order $x_1, \dots , x_8$ is chosen for the eight points of the regular Cayley octad $X$ of the corresponding net $\Lambda_X$, the projection from the point $x_8$ onto $\P^2$ sends $x_1, \dots , x_7$ to a set of points $y_1, \dots , y_7$ and the Steiner curve $\Gamma$ to a sextic with seven double points at $y_1, \dots , y_7$. The images of the exceptional curves blown up from these points $y_1, \dots , y_7$ are the seven bitangents corresponding to the lines $L_{i8}$ joining $x_i$ and $x_8$; this set of bitangents is actually an Aronhold system for the curve $C$.\\

The algorithm described in \cite{PSV} is meant to compute the matrix $A(z)$ for a non singular plane quartic $C$, described by the equation:
\bes
f(z_1,z_2,z_3)=c_{400}z_1^4+c_{310}z_1^3z_2+c_{301}z_1^3z_3+c_{220}z_1^2z_2^2+c_{211}z_1^2z_2z_3+\dots +c_{004}z_3^4,
\ees
where $c_{ijk}$ are the 15 coefficients of the quartic.
A determinantal representation of $C$ is obtained in terms of 3 suitable $4\times 4$ symmetric matrices as follows:
\bes
f(x,y,z)=\det(z_1A_1+z_2A_2+z_3A_3)\equiv \det(A(z)),
\ees
and $A_1$, $A_2$ and $A_3$ are the matrices of the bilinear forms associated to three quadrics $Q_1$, $Q_2$ and $Q_3$ in $\P^3$. The algorithm determines $A(z)$ from the bitangents of the curve. The 28 bitangents $b_m(\tau, z)$ of the plane quartic $C$ correspond to the $28$ gradients of the odd theta functions, as they are related to the first term of the Taylor expansion of the theta functions with odd characteristics $\theta_m(\tau, z)$. A $4\times 4$ matrix $V$ is built after a choice of three bitangents among the $\binom{28}{3}=3276$ possible choices. If the triple of bitangents is not a subtriple of an Aronhold system, then the determinant of the resulting matrix $V$ is identically null. Hence, the algorithm needs to start from a triple of bitangents that is contained in an Aronhold system; the number of these triples is 2016 and the corresponding triples of gradients of odd theta functions are called azygetic. The other $3276-2016=1260$ triples are known as syzygetic.
Once an azygetic triple is fixed, the matrix $A(z)$ is given by the adjoint of $V$ divided by $f^2$.
These 2016 determinantal representations factorize into 36 equivalence classes (two representations $A(z)$ and $A'(z)$ are equivalent whenever they are conjugated under the action of $\rm{GL_4}$, i.e $A'(z)=U^tA(z)U$, for $U\in\rm{GL_4}$).\\
The bitangent matrix originates from $A(z)$. We can consider the $8\times 4$ matrix given by the coordinates of the eight points of the Cayley octad:
\[
\texttt{X}:=\begin{pmatrix}x_{10}&x_{11}&x_{12}&x_{13} \\
x_{20}&x_{21}&x_{22}&x_{23} \\
\vdots&\vdots&\vdots&\vdots\\
x_{80}&x_{81}&x_{82}&x_{83}
\end{pmatrix},
\]
and the  $8\times 8$ symmetric matrix:
\begin{equation}\label{E:hesse2}
L_X(z) := \texttt{X}\ A(z)\ ^t\texttt{X}.
\end{equation}
Clearly, $L_X(z)$ is a matrix of rank 4 with zero entries on the main diagonal. The 28 entries of $L_X(z)$ off the main
diagonal are linear forms in $z$ that define the bitangents
of $C$, as seen in (\ref{bitangent}).
Notice that the determinantal representations given by each of the  $\binom{8}{4}=70$ principal $4\times 4$ minors of the matrix represent the same quartic and lie in the same equivalence class.
Once a representation is determined, the other 35 inequivalent representations can be obtained by acting properly on the corresponding Cayley octad. Each of the $2016=56\cdot 36$ azygetic triples appears as a product of the corresponding bitangents in exactly one of the $\binom{8}{3}=56$ principal $3\times 3$ minor of one of the 36 inequivalent bitangent matrices; thus these minors are parametrized by azygetic triples.

In this paper we tackle the inverse problem, i.e. the construction of the   bitangent matrix   starting from the 28 bitangents  of the plane quartic. We need to start by a suitable $8 \times 8$ matrix. Since we are interested in azygetic triples in order to obtain the   bitangent matrix, we resort to Aronhold systems to build such a matrix.
For any fixed even characteristic, there exist eight corresponding Aronhold sets, see Section \ref{aronhold}, ($288=36\cdot8$), which are obtained by translation from a chosen one. These $7$-tuples of odd characteristics and the even one will be the rows of the matrix, whose $3\times 3$ principal minors will contain 56 distinct azygetic triples ($2016=36\cdot 56$). Thus we will work with the following bitangent matrix:
\bes 
\label{m1}
\mathcal{M}:=
\begin{pmatrix}
0 & b_{77} & b_{64} & b_{51} & b_{46} & b_{23} & b_{15} & b_{32} \\
b_{77} & 0 & b_{13} & b_{26} & b_{31} & b_{54} & b_{62} & b_{45} \\
b_{64} & b_{13} & 0 & b_{35} & b_{22} & b_{47} & b_{71} & b_{56} \\
b_{51} & b_{26} & b_{35} & 0 & b_{17} & b_{72} & b_{44} & b_{63} \\
b_{46} & b_{31} & b_{22} & b_{17} & 0 & b_{65} & b_{53} & b_{74} \\
b_{23} & b_{54} & b_{47} & b_{72} & b_{65} & 0 & b_{36} & b_{11} \\
b_{15} & b_{62} & b_{71} & b_{44} & b_{53} & b_{36} & 0 & b_{27} \\
b_{32} & b_{45} & b_{56} & b_{63} & b_{74} & b_{11} & b_{27} & 0 \\
\end{pmatrix},
\ees
see Section \ref{bitangmatrix} for the explanation of the  meaning  of the indeces.\smallskip

Notice that generally this matrix has rank eight,
so one  has to  determine suitable coefficients $c_{ij}$ in  such  a way that the matrix
$ (c_{ij} b_{ij})$ has rank  four.
The aim of this paper is to determine uniquely such coefficients,   up to  congruences  by diagonal matrices, once an even characteristic $m$  and  a compatible  Aronhold set of characteristics (i.e  a level  two structure of the  moduli space of  principally polarized abelian  varieties of  genus 3) are given. As multilinear algebra techniques are not sufficient to determine these coefficients, we will also need to carefully use Riemann's relations and Jacobi's formula.\\
The coefficients will turn out to be modular functions holomorphic along the  locus of the period matrices of smooth plane quartics (cf. Theorem \ref{main}).
Any other inequivalent bitangent matrix will be obtained by changing the even characteristic and considering the eight corresponding Aronhold sets.

\medskip
\section {Aknowledgments}
We are grateful to  Giorgio Ottaviani  and Edoardo Sernesi  for  explaining to us the construction of the  bitangent  matrix and the  beautiful geometry  related to  plane quartics.  The first three sections of this  paper have  been strongly influenced  by discussions that  the senior author had with them. We are also grateful to Igor V. Dolgachev for pointing out inaccuracies in the first version of the paper. 

\section{Aronhold systems.}  \label{aronhold}
In this section we introduce Aronhold systems of bitangents and Aronhold sets of theta characteristics and recall some basic facts about characteristics and the action of the symplectic group on them.

The next definition is of central
importance in the geometry of plane quartics. Let $C$ be a
nonsingular plane quartic.

\begin{definition}\label{D:aron1}
A $7$-tuple $\{\ell_1,\dots, \ell_7\}$  of bitangent lines to $C$
is called an {\rm Aronhold system of bitangents} if for each
triple $\ell_i,\ell_j,\ell_k$ the six points of contact of
$\ell_i\cup \ell_j\cup \ell_k$ with $C$ are not on a conic.
\end{definition}

Not all $7$-tuples of bitangents are Aronhold systems. There are
exactly $288$ Aronhold systems among the ${28\choose 7}$
$7$-tuples of bitangents of $C$ (for more details we refer to \cite{Do})  Denote by $\theta_i$ the effective half-canonical divisor
such that $2\theta_i = C\cap \ell_i$ (i.e. $\theta_i$, or
$\mathcal{O}(\theta_i)$, is  an odd theta-characteristic). The
condition that $\{\ell_1,\dots, \ell_7\}$ is an Aronhold system is
equivalent to the condition that for each triple of pairwise
distinct indices $i, j, k$ we have
$$
|2K- \theta_i - \theta_j-\theta_k| = \emptyset
$$
and replacing $2K$ by $2\theta_i+2\theta_j$ the condition is seen
to be equivalent to the following: $\theta_i+\theta_j-\theta_k$ is
an even theta-characteristic for each $i\ne j\ne k$. This can
be taken as another definition of Aronhold system (cf. also Definition \ref{E:fundsys}).

\begin{definition}\label{D:aron2}
\emph{An Aronhold set} on a non-hyperelliptic curve $C$ of
genus 3 is a 7-tuple $\{\theta_1, \dots, \theta_7\}$ of distinct
odd theta-characteristics such that $\theta_i + \theta_j-\theta_k$
is an even theta-characteristic for each triple of pairwise
distinct indices $i, j, k$.
\end{definition}

\ni We  have a purely  combinatoric interpretation of  the above description.\smallskip

 \ni A characteristic $m$ is a column
vector in $\Z^{2g}$, with $m'$ and $m''$ as first and second
entry vectors. If we set:
\be \label{eq2}
e(m)=(-1)^{^{t}m'm''},
\ee
then
$m$
is called even or odd according as $e(m)=1$ or $-1$.
For any triplet $m_1, m_2, m_3$ of characteristics we set
\be \label{eq3}
e(m_1, m_2, m_3)=e(m_1)e(m_2)e(m_3)e(m_1+m_2+m_3).
\ee
A sequence $m_1,\dots,m_r$ of characteristics is essentially independent
if for any choice of an even number of indeces between $1$ and $r$ the sum of the corresponding characteristics is not congruent to
$0\, \rm{mod}\,2$.

The unique action of $\Gamma_g= {\rm Sp}( 2g, \Z)$ on the set of characteristics $\rm{mod}\,
2$ keeping invariant \eqref{eq2}, \eqref{eq3} and the condition of being essentially
independent is defined by

$$\sigma\cdot m:=\begin{pmatrix}
 D &  -C \cr
 -B  & A \end{pmatrix}
\begin{pmatrix}
m'\cr
m''
\end{pmatrix} + \begin{pmatrix} diag(C^t D)\cr
diag(A^t B) \end{pmatrix} .$$
Henceforward, we often shall consider characteristics with $0$ and
$1$
as entries. In this situation a special role is played by
sequences of characteristics that form a  fundamental system, defined as follows. 
\begin{definition}\label{E:fundsys}
A sequence of $2g+2$ characteristics  in $\F_2^{2g}$  is   a fundamental system if all triplets are azygetic, i.e   
$$e(m_i, m_j, m_k)=-1,$$ 
for all indeces $1\leq i<j<k\leq 2g+2$.
\end{definition}

Fundamental systems exist and are all conjugate
under an extension of $\Gamma_g$ by translations, we  refer to \cite{ig80} and \cite{fayriemann}  for details.
The  number of odd characteristics in a fundamental system is always congruent to $g\, \rm{mod}\, 4.$
So when $g=3$ we have fundamental systems with $3$ or $7 $  odd characteristics.\bigskip

From now on we fix our attention on genus 3 case. 

\begin{definition}\label{E:fundsys}
Let  $$m_0, n_1, n_2, \dots, n_7$$ be a fundamental system with one  even characteristic, $m_0$, and seven odd characteristics. In this case $$ n_1, n_2, \dots, n_7$$
is called an Aronhold set and necessarily
$m_0=\sum_{i=1}^7 n_i$.
\end{definition}\medskip
  
There are exactly $288= 36\cdot 8$ such systems, hence each even characteristic appears exactly in eight such fundamental systems.  We remark that the ordered set of fundamental systems are 
$$288\cdot 7!= 36\cdot 8!=\vert {\rm Sp}( 6, \F_2)\vert.$$
Concerning these fundamental systems with a  fixed  even  characteristic $m_0$, we have the following Lemma.
\begin{lemma}
A fundamental system $m_0, n_1, n_2, \dots, n_7$ determines the remaining 7 via translations.\end{lemma}

\begin{proof}
The other seven fundamental systems can be obtained translating the initial fundamental system  $m_0, n_1, n_2, \dots, n_7$  with
$m_0+n_i$ with $i=1,\dots, 7$.
\end{proof}
 
  \begin{remark}{\label{E:matr}} We observe that the  $8\times 8$ matrix
 
$$
\begin{pmatrix}
  m_0&n_1&n_2&\dots& n_7\\
 n_1& m_0&(m_0+n_1)+ n_2& \dots & (m_0+n_1)+ n_7\\
 n_2& (m_0+n_2)+ n_1&m_0& \dots & (m_0+n_2)+ n_7\\
\dots & \dots &\dots & \dots&\dots \\
    n_7& (m_0+n_7)+ n_1&\dots & \dots& m_0\\
\end{pmatrix}
$$
   is symmetric in the symbols.
    It is unique once we fix  a row, up to permutation of rows (and corresponding symmetric permutation of columns).    
Here is an explicit example  (we use a  row notation):
$$
\begin{pmatrix}
[000,000]&[111,111]&[110,100]&[101,001]&[100,110]&[010,011]&[001,101]&[011,010]\\
{[111,111]}&[000,000]&[001,011]&[010,110]&[011,001]&[101,100]&[110,010]&[100,101]\\
{[110,100]}&[001,011]&[000,000]&[011,101]&[010,010]&[100,111]&[111,001]&[101,110]\\
{[101,001]}&[010,110]&[011,101]&[000,000]&[001,111]&[111,010]&[100,100]&[110,011]\\
{[100,110]}&[011,001]&[010,010]&[001,111]&[000,000]&[110,101]&[101,011]&[111,100]\\
{[010,011]}&[101,100]&[100,111]&[111,010]&[110,101]&[000,000]&[011,110]&[001,001]\\
{[001,101]}&[110,010]&[111,001]&[100,100]&[101,011]&[011,110]&[000,000]&[010,111]\\
{[011,010]}&[100,101]&[101,110]&[110,011]&[111,100]&[001,001]&[010,111]&[000,000]\\
\end{pmatrix}.$$
Notice that 36 essentially different such matrices can be constructed, corresponding to the 36 even characteristics. 
\end{remark}
In the next section we will show how this matrix can be used to build the bitangent matrix by means of gradients of odd theta functions. 

\section{Theta Functions} \label{thetafunct}

 We intend  to give an   explicit analytic  expression for the  bitangents.  The  main tool will be theta functions.
We denote by $\H_g$ the {\it Siegel upper half-space}, i.e. the
space of complex symmetric $g\times g$ matrices with positive definite
imaginary part. An element $\tau\in\H_g$ is called a {\it period
matrix}, and defines the complex abelian variety $X_\tau:=\C^g/\Z^g+\tau
\Z^g$. The group $\Gamma_g:={\rm Sp}(2g,\Z)$ acts on $\H_g$ by
automorphisms. For $$\gamma:=\begin{pmatrix} a&b\\
c&d\end{pmatrix}\in{\rm Sp}(2g,\Z)$$ the action is
$\gamma\cdot\tau:=(a\tau+b)(c\tau+d)^{-1}$. The quotient of $\H_g$ by
the action of the symplectic group is the moduli space of
principally polarized abelian varieties (ppavs): $\A_g:=\H_g/{\rm
Sp}(2g,\Z)$. The case $g=1$ is special and in the following we 
will always assume $g>1$.

We define the {\it level} subgroups of the symplectic group to be
$$
\Gamma_g(n):=\left\lbrace\gamma=\begin{pmatrix} a&b\\ c&d\end{pmatrix}
\in\Gamma_g\, |\, \gamma\equiv\begin{pmatrix} 1&0\\
0&1\end{pmatrix}\ {\rm mod}\ n\right\rbrace,
$$
$$
\Gamma_g(n,2n):=\left\lbrace\gamma\in\Gamma_g(n)\, |\, {\rm diag}(a^tb)\equiv{\rm diag}
(c^td)\equiv0\ {\rm mod}\ 2n\right\rbrace.
$$
The corresponding {\it level moduli spaces of ppavs} are denoted
$\A_g(n)$ and $\A_g(n,2n)$, respectively.

A holomorphic  function $F:\H_g\to\C$ is called a {\it modular form of weight $k$ with
respect to $\Gamma\subset\Gamma_g$} if
$$
F(\gamma\cdot\tau)=\det(c\tau+d)^kF(\tau),\quad \forall \gamma=
\begin{pmatrix}a&b\\ c&d\end{pmatrix}\in\Gamma,\ \forall \tau\in\H_g.
$$
More generally, let $\rho:{\rm GL}(g,\C)\to\operatorname{End} V$ be some representation.
Then a map $F:\H_g\to V$ is called a {\it $\rho$- or $V$-valued modular form}, or,
if there is no ambiguity in the choice of $\rho$,
simply
a {\it vector-valued modular form}, 
with
respect to $\Gamma\subset\Gamma_g$ when
$$
F(\gamma\cdot\tau)=\rho(c\tau+d)F(\tau),\quad \forall \gamma=
\begin{pmatrix} a&b\\ c&d\end{pmatrix}\in\Gamma,\
\forall \tau\in\H_g.
$$

For $m',m''\in \Z^g$ and $z\in \C^g$ 
we define the {\it   theta function with characteristic} $m= \left[{}^{m'}_{m''} \right]$ to be
%
%
$$
\tt {m'}{m''}(\tau,z):=\theta_m(\tau, z):=\sum\limits_{p\in\Z^g} \exp \pi i \left[\left(
p+\frac{m'}{2},\tau(p+\frac{m'}{2})\right)+2\left(p+\frac{m'}{2},z+
\frac{m''}{2}\right)\right],
$$ 
where $(\cdot,\cdot)$ at the exponent denotes the usual scalar product.
Here we list some properties of the theta functions. 
First, we observe that
\bes
\theta_{m+2n}(\tau, z)=(-1)^{^t m ' n''}\theta_m(\tau, z), \qquad n\in\Z^{2g}.
\ees
Hence, the theta functions with characteristics can be parametrized by $2^{2g}$
vector columns $m'$, $m''$ with $m'$ and $m''$ thought as entries in $\{0,1\}^g$. Note that these are the roots of the canonical bundle. The preceding formula is called {\it reduction formula}.
Henceforward, we refer to such characteristics as  {\it reduced characteristics} and to the corresponding theta functions as theta functions with  half integral characteristics; clearly all the properties stated in Section \ref{aronhold} also hold in this case.
Then, we recall the behavior of the theta functions under a change of sign of the $z$ variable:
$$\tt {m'}{m''}(\tau,-z)=e(m)\tt {m'}{m''}(\tau,z).$$
The following formula shows that adding a so called {\it half period}, $\tau\frac{m'}{2}+\frac{m''}{2}$, to the argument $z$ actually permutes the functions with  half integral characteristics (see \cite{ig1} or \cite{rafabook}):
$$
\tt {m'}{m''}(\tau,z)=\exp{(\pi i\left[\left(\frac{m'}{2},\tau\frac{m'}{2} \right) +2\left(\frac{m'}{2},z+\frac{m''}{2} \right)\right])}\tt {0}{0}(\tau,z+\tau\frac{m'}{2}+\frac{m''}{2}).
$$

The {\it  reduced characteristic} $m$ is called {\it even} or {\it
odd} depending on whether the scalar product $m'\cdot m'' \in\Z_2$ is
zero or one and the corresponding  theta function is even or odd
in $z$, respectively. The number of even (resp. odd) theta characteristics is
$2^{g-1}(2^g+1)$ (resp. $2^{g-1}(2^g-1)$).
The transformation law for theta functions under the action of the symplectic
group is (see \cite{ig1}): 
$$
\tt{m'}{m''}(\tau,z)=\phi({m'},{m''},\gamma,\tau,z)\det
(c\tau+d)^{1/2}\theta\left[(^t\gamma^{-1})\begin{pmatrix}{m'}\\ {m''}
\end{pmatrix}\right](\gamma\cdot\tau,(c\tau+d)z),
$$
where $\phi$ is some complicated explicit function, and the action of
$^t\gamma^{-1}$ on characteristics is taken modulo integers. It is further
known (see \cite{ig1}, \cite{sm1}) that for $\gamma\in\Gamma_g(4,8)$ we
have $\phi|_{z=0}=1$, while $^t\gamma^{-1}$ acts trivially on the
characteristics $m$. Thus the values of theta functions at $z=0$, called {\it theta constants}, are modular forms of weight
one half with respect to $\Gamma_g(4,8)$. We will denote them with $\theta_m$.

The group $\Gamma_g(2)/
\Gamma_g(4,8)$ acts on the set of theta-constants $\theta_m$ by
certain characters whose values are fourth roots of the unity, as shown in \cite{sm1}.
The action of
$\Gamma_g/\Gamma_g(2)$ on the set of theta  with half integral characteristics  is by
permutations. Since the group $ \Gamma_g(1,2)$ fixes the null characteristic, it acts on $\theta_0$ by a multiplier.

All odd theta constants  with half integral characteristics vanish identically, as the
corresponding theta functions are odd functions of $z$, and thus there
are $2^{g-1}(2^g+1)$ non-trivial theta constants.

Differentiating the theta transformation law above with respect 
to different $z_i$ and then evaluating at $z=0$, we see that for $\gamma\in
\Gamma_g(4,8)$ and $m=\left[{}^{m'}_{m''}\right]$ odd
$$
\frac{\p}{\p z_i}\tt{m'}{m''}(\tau,z)|_{z=0}=\det(c\tau+d)^{1/2}\sum\limits_j
(c\tau+d)_{ij}\frac{\p}{\p z_j}\tt{m'}{m''}(\gamma\cdot\tau,(c\tau+d)z)|_{z=0},
$$
in other words the gradient vector $\rm{grad}_z\,
\tt{m'}{m''}(\tau,0)$ is a $\mathbb C^g$-valued modular 
form with respect to $\Gamma_g(4,8)$ under the representation 
$\rho(M)=(\det M)^{1/2} M$, for $M\in\Gamma_g(4,8)$.

The set of all even theta constants defines the map
$$
\P \rm{Th}:\A_g(4,8)\to \P^{2^{ g-1}(2^g+1)-1}, \qquad \bar{\tau}\mapsto [\cdots,\theta_m(\tau),\cdots],
$$ 
with $\bar{\tau}\in A_g(4,8)=\H_g/\Gamma_g(4,8)$ and $\tau$ a representative of the equivalence class $\bar{\tau}$.
 It is known that the map $\P \rm{Th}$ is injective, see \cite{ig1} and
 references therein. 
 Considering the set of gradients of all odd theta
 functions at zero gives the map$$
 \rm{grTh}:\H_g\to (\C^g)^{\times 2^{g-1}(2^g-1)},\qquad \tau\mapsto\rm{grTh}(\tau):=\left\lbrace
 \cdots,grad_z\tt{m'}{m''},\cdots\right\rbrace_{{\rm all\ odd}\, m},$$
 which due to modular properties descends to the quotient
 $$\P \rm{grTh}:\A_g(4,8)\to  (\C^g)^{\times 2^{g-1}(2^g-1)}/\rho({\rm GL}(g,\C)),$$
 where ${\rm GL}(g,\C)$ acts simultaneously on all $\C^g$'s in the product by $\rho$.

 The image of $\P
 \rm{grTh}$ actually lies in the Grassmannian,
 \begin{equation*}
 \P \rm{grTh}:\A_g(4,8)\to{\rm Gr}_\C(g,\,2^{g-1}(2^g-1))
 \end{equation*}
 of $g$-dimensional subspaces in $\C^{2^{g-1}(2^g-1)}$. The Pl\"ucker
 coordinates of this map are modular forms of weight $\frac{g}{2}+1$ and
 have been extensively studied, see \cite{ig2,sm3,gsm}. Moreover, in \cite{gsm}, it is implicitely proved the following proposition.  
\begin{prop}
The map
\begin{equation*}
 \P \rm{grTh}:\A_3(4,8)\to{\rm Gr}_\C(3,\, 28)
 \end{equation*}
is  injective.
\end{prop}

 
In genus 3 case the evaluation at  zero of $\tt{m'}{m''}(\tau,z)$ and of the gradients have a  significative meaning.
In fact, as consequence of Riemann singularity theorem the following proposition holds.
 \begin{prop} Let $\tau$  be a jacobian matrix,  then it is the period matrix of a hyperelliptic jacobian if and only if there exist an even characteristic $m$ with  $\tt{m'}{m''}(\tau,0)=0$.
  
Let $\tau$  be the   period matrix of a non-hyperelliptic jacobian (i.e the jacobian of a plane quartic), then for all odd characteristics $m$  the gradient vector $\rm{grad}_z \tt{m'}{m''}(\tau,0)$  parametrizes the bitangents of the plane quartic.
\end{prop}

\begin{remark} The equations of the bitangents are
$$b_m(\tau, z):=\frac{\partial\tt{m'}{m''}(\tau,z)}{\partial z_1}\vert_{z=0}z_1+\frac{\partial\tt{m'}{m''}(\tau,z)}{\partial z_2}\vert_{z=0}z_2+\frac{\partial\tt{m'}{m''}(\tau,z)}{\partial z_3}\vert_{z=0}z_3=0.$$

From now on, if it will be necessary, we  will identify the gradient vectors with the bitangents
\end{remark}

The following corollary is easily derived.
\begin{corollary}
The hyperelliptic locus $\mathcal I_3 \subset \mathcal A_3$  is defined by the equation   
$$\prod_{m\, even}\tt{m'}{m''}(\tau,0)=0.$$
\end{corollary}

\section{The bitangent matrix} \label{bitangmatrix}  
This section will be entirely devoted to the description of the bitangent matrix related to the plane quartic. First of all, we will see how the language introduced in the previous section suitably translates the combinatorics described in Section \ref{aronhold} so as to let us build the matrix. Then, we will recall the properties of such a matrix and how it completely describes the geometry of the curve.

Using the language of Section \ref{thetafunct}, the geometric condition defining an Aronhold system can be rephrased as a combinatorial condition, as $\theta_i+\theta_j- \theta_k$
  is an even theta characteristic whenever
  $m_i+m_j+m_k$ is an even characteristic. So we can apply what we stated in Remark \ref{E:matr}  to the $8\times 8$ symmetric matrix
$L_X(z)$, cf. \eqref{E:hesse2}, since it is related to  the matrix obtained using the derivatives of odd theta functions ordered as in Remark \ref{E:matr}.  More precisely, using the notations in Remark \ref{E:matr}
we set
$$
M_X(z):=\begin{pmatrix} 
0 &b_{n_1}(\tau,z) &\dots &b_{n_7}(\tau,z)\\
b_{n_1}(\tau,z)  & 0  &\dots &b_{n_7+m_0+n_1}(\tau,z)\\
\dots & \dots &0&\dots \\
b_{n_7}(\tau,z)&b_{n_1+m_0+n_7}(\tau,z) & \dots &0\\
\end{pmatrix},
$$
and
$$
M_i:=
\begin{pmatrix}
0 & \frac{\partial}{\partial z_i}\theta_{n_1}(\tau,z)|_{z=0} & \cdots & \frac{\partial}{\partial z_i}\theta_{n_7}(\tau,z)|_{z=0} \\
\frac{\partial}{\partial z_i}\theta_{n_1}(\tau,z)|_{z=0} & 0 & \cdots & \frac{\partial}{\partial z_i}\theta_{n_7+m_0+n1}(\tau,z)|_{z=0} \\
\cdots & \cdots & 0 & \cdots \\
\frac{\partial}{\partial z_i}\theta_{n_7}(\tau,z)|_{z=0} & \frac{\partial}{\partial z_i}\theta_{n_1+m_0+n_7}(\tau,z)|_{z=0} & \cdots & 0 \\
\end{pmatrix}.
$$
Thus the matrices
$$
M_X(z) = z_1 M_1+ z_2 M_2+ z_3 M_3
$$
and $L_X(z) $ are related. In both cases the entries of the matrices  $M_X(z)$  and  $L_X(z) $  are the bitangents to the canonical curves, but  they are not uniquely determined, (entries can differ by  different proportionality factors) so the matrix  $M_X(z)$ has not necessarily rank 4.  



 


We will manipulate the matrix $M_X(z)$ and determine suitable coefficients $c_{ij}$ in order that the matrix $(c_{ij}M_X(z)_{ij})$ has rank four.
We will do it using the action of the symplectic group, Riemann theta formula and Jacobi's derivative  formula. We  recall shortly them in the  genus 3 case.

For any triple  $n_1, n_2, n_3$ of odd characteristics we   set
$$D(n_1, n_2,n_3)(\tau):=\rm{grad}_z\,
 \theta_{n_1}(\tau,0)\wedge \rm{grad}_z\,
 \theta_{n_2}(\tau,0)\wedge \rm{grad}_z\,
 \theta_{n_3}(\tau,0).
 $$
We  recall when  such  nullwerte of jacobian of theta functions is a polynomial in the  theta constants. The following statement holds,  see \cite{IgNu}.
\begin{prop}  Suppose that $g=3$ and   $n_1, n_2, n_3$ are odd  characteristics distinct  mod 2; then 
$D(n_1, n_2, n_3)$ is  a polynomial in the theta constants   if and only if  $n_1, n_2,n_3$  form an azygetic triplet. Moreover, we have
$$
D(n_1, n_2, n_3)(\tau) =\pm  (\pi)^3\theta_{m_1}\theta_{m_2}\theta_{m_3}\theta_{m_4}\theta_{m_5}(\tau),
$$
with  $m_1,\dots, m_5$ even characteristics  and $ n_1, n_2, n_3, m_1,\dots, m_5$ a uniquely determined fundamental system of characteristics.
\end{prop}

Also, we briefly recall that Riemann's quartic addition theorem for theta constants with characteristic in genus three has the form
\begin{equation*}\label{Riemann}
  r_1=r_2+r_3,
\end{equation*}
where each $r_i$ is a product of four theta constants with characteristics forming an even coset of a two-dimensional isotropic space. Such isotropic spaces are constructed by means of the symplectic form on $\F_2^6$  defined  by
$$
e(m, n):=(-1)^{m'^tn''-m''^tn'}.
$$




%
%

A full description of the curve C is provided by the matrix $L_X(z)$ defined in \eqref{E:hesse2}, as the following notable proposition states.
\begin{prop}\label{P:sturm1}
Let $L_X(z)$ be the matrix defined in \eqref{E:hesse2}, then:
\begin{itemize}
\item[(i)] given the net $\Lambda_X$ and a basis $\{Q_1,Q_2,Q_3\}$
of it, the matrix $L_X(z)$ is uniquely defined up to
simultaneous multiplication by a constant factor of a row and the
corresponding column and up to simultaneous permutations of rows
and columns;

\item[(ii)] the 28 entries of $L_X(z)$ outside the main
diagonal are linear forms in $z$ that define the bitangents
of $C$;

\item[(iii)] the seven bitangents on a given row (column) are
elements of an Aronhold system. The 8 Aronhold systems represented
by the rows (columns) of $L_X(z)$ are associated to the even
theta characteristic on $C$ defined by the net $\Lambda_X$;

\item[(iv)] $L_X(z)$ has identically rank 4, and \emph{any}
of its $4\times 4$ minors is a polynomial of degree 4 in $z$
which defines $C$.

\end{itemize}
\end{prop}
\begin{proof}
${}$
\begin{itemize}
\item[(i)]
A change of the order of the eight points of the Cayley octad $x_1,\ldots,x_8$corresponds to a simultaneous permutation of rows and columns for the matrix. A change of the homogeneous coordinates of these points corresponds to a simultaneous multiplication by a constant factor. Hence, the statement follows.
\item[(ii)]
We refer to \cite{PSV} for details (cf. Section 1, eq. (1.3)).
\item[(iii)]
The statement is a consequence of what we explained at the beginning of this section.
\item[(iv)]
The claim obviously follows by the way the matrix was defined.
\end{itemize}
\end{proof}

\section{Determining  analytically the  bitangent matrix}
Now we want to determine analytically the bitangent matrix; this will be a partial converse of the Proposition \ref{P:sturm1}. \smallskip

\noindent Our initial  datum  will be  the $28$  gradients of  odd theta functions evaluated at $0$, corresponding to the bitangents and a chosen even characteristics $m$. 
Therefore, to obtain a matrix congruent to $L_X(z)$,
 we have to determine the  values of the  functions
$c_{n_i }(\tau)$ in the  matrix:
$$
\begin{pmatrix}
  0&c_{n_1}(\tau)
 b_{n_1}(\tau, z)&c_{n_2}(\tau)
 b_{n_2}(\tau, z)&\dots&c_{n_7}(\tau)
 b_{n_7}(\tau, z)\\
c_{n_1}(\tau)
 b_{n_1}(\tau, z)&0&c_{m+n_1+ n_2}(\tau)
 b_{m+n_1+n_2}(\tau, z)& \dots& c_{m+n_1+n_7}(\tau)
 b_{m+n_1+n_7}(\tau, z)\\
c_{n_2}(\tau)
 b_{n_2}(\tau, z)& c_{m+n_1+n_2}(\tau)
 b_{m+n_1+n_2}(\tau, z)&0& \dots&c_{m+n_2+n_7}(\tau)
 b_{m+n_2+n_7}(\tau, z)\\
\dots & \dots &\dots & \dots&\dots \\
 c_{n_7}(\tau)
 b_{n_7}(\tau, z)& c_{m+n_1+n_7}(\tau)
 b_{m+n_1+n_7}(\tau, z)&\dots& \dots& 0\\
\end{pmatrix}.
$$

\ni There are, up to  permutations, 8 Aronhold sets whose sum is the given characteristic.  The subgroup  of the  symplectic group  fixing  the characteristic $m$,  permutes the element of a fixed  Aronhold set and the  eight Aronhold sets. 
Now  we choose an  even  characteristic  and  an  Aronhold set. As we wrote, the  number of possibilities is exactly  equal to $\vert Sp(6, \F_2)\vert$. 
To  simplify our  computation we     assume $m=0$ and  as  Aronhold set  we use the one in Remark\,\ref{E:matr}.

We  want  to  obtain a matrix of rank  four.  Because of the action of the   symplectic  group, we can assume that    all columns are linear combination of the first  four. This  will be our starting hypothesis. We recall that any  set of  four bitangents coming from an  Aronhold set  form a  fundamental system of $\P^2$.
We need seven  bitangents $b_1,\dots b_7$  forming an Aronhold set. We  choose those in   Section \ref{aronhold}.  
Let $\mathcal{M}$ be the symmetric matrix:
\bes 
\label{m1}
\mathcal{M}:=
\begin{pmatrix}
0 & b_{77} & b_{64} & b_{51} & b_{46} & b_{23} & b_{15} & b_{32} \\
b_{77} & 0 & b_{13} & b_{26} & b_{31} & b_{54} & b_{62} & b_{45} \\
b_{64} & b_{13} & 0 & b_{35} & b_{22} & b_{47} & b_{71} & b_{56} \\
b_{51} & b_{26} & b_{35} & 0 & b_{17} & b_{72} & b_{44} & b_{63} \\
b_{46} & b_{31} & b_{22} & b_{17} & 0 & b_{65} & b_{53} & b_{74} \\
b_{23} & b_{54} & b_{47} & b_{72} & b_{65} & 0 & b_{36} & b_{11} \\
b_{15} & b_{62} & b_{71} & b_{44} & b_{53} & b_{36} & 0 & b_{27} \\
b_{32} & b_{45} & b_{56} & b_{63} & b_{74} & b_{11} & b_{27} & 0 \\
\end{pmatrix},
\ees
where $b_{ij}:= \sum_{k=1}^{3}(\frac{\partial}{\partial z_k}\theta_{ij}|_{z=0})z_k$, $\theta_{ij}$ being the theta function associated with the odd characteristic  $[{}^{i}_{j}]:=[{}^{a_1 \, a_2 \, a_3}_{b_1 \, b_2 \,\, b_3}]$, where $i=a_1 2^2+a_2 2+a_3$ and $j=b_1 2^2+b_2 2+b_3$. Each equation $b_{ij}=0$ defines a bitangent. Clearly these bitangents do not change by multiplying each entry of the matrix by a function which does not depend on the variables $z_k$, with $k=1,2,3$. 
\smallskip
\\
To find out what these $28$ coefficients are, we will resort to the following procedure. We will first 
determine the coefficients of suitable $5 \times 5$ principal minors so as to get symmetric matrices of rank $4$; this will define some relations among the column vectors of the matrix. Then we will act on these minors properly so as to make their common entries equal.  Finally we will use the resulting relations among the column vectors to determine the remaining coefficients of the matrix.\\
\\
We first focus on the submatrix obtained by taking the first five rows and the first five columns. We need to compute $\lambda_{ij}$ such that:
\be
\label{A5}
\rm{rank}\begin{pmatrix}
0 & \lambda_{77}b_{77} & \lambda_{64}b_{64} & \lambda_{51}b_{51} & \lambda_{46}b_{46} \\
\lambda_{77}b_{77} & 0 & \lambda_{13}b_{13} & \lambda_{26}b_{26} & \lambda_{31}b_{31} \\
\lambda_{64}b_{64} &\lambda_{13} b_{13} & 0 & \lambda_{35}b_{35} & \lambda_{22}b_{22} \\
\lambda_{51}b_{51} & \lambda_{26}b_{26} & \lambda_{35}b_{35} & 0 & \lambda_{17}b_{17} \\
\lambda_{46}b_{46} & \lambda_{31}b_{31} & \lambda_{22}b_{22} & \lambda_{17}b_{17} & 0 \\
\end{pmatrix}=4.
\ee
\\
\ni Note that $\rm{rk}(D^tAD)=\rm{rk}(A)$ for any invertible diagonal matrix $D$. Therefore, the matrix in (\ref{A5}) is determined up to an invertible diagonal matrix which acts by congruence. \\
\\
\ni The condition of linear dependence on the vector columns $V_i$ of the matrix in (\ref{A5}):
\bes
\alpha_1V_1+\alpha_2V_2+\alpha_3V_3+\alpha_4V_4+\alpha_5V_5=0
\ees
\ni can be turned into:
\be
\label{ldcondition}
\tilde{V}_1+\tilde{V}_2+\tilde{V}_3+\tilde{V}_4-\tilde{V}_5=0
\ee
\ni when both the sides of the matrix are multiplied by the diagonal matrix $\rm{diag}\left( 
\alpha_1^{-1},\alpha_2^{-1},\alpha_3^{-1}, \alpha_4^{-1}, -\alpha_5^{-1} 
\right)$.
Hence, we can compute the coefficients $\lambda_{ij}$ by demanding the condition (\ref{ldcondition}) without any loss of generality. Note that whenever such an operation is performed again on the matrix, the diagonal matrix on the left will change the coefficients in (\ref{ldcondition}).\\
On the first row (\ref{ldcondition}) leads to:
\bes
\lambda_{77}b_{77} + \lambda_{64}b_{64} + \lambda_{51}b_{51} = \lambda_{46}b_{46},
\ees
which is equivalent to a linear system of three equations in the variables $\lambda_{77}$, $\lambda_{64}$, $\lambda_{51}$, $\lambda_{46}$:
\be
\label{system}
\left \{ \begin{array}{l}
\lambda_{77}\partial_1\theta_{77}|_{z=0} + \lambda_{64}\partial_1\theta_{64}|_{z=0} + \lambda_{51}\partial_1\theta_{51}|_{z=0} = \lambda_{46}\partial_1\theta_{46}|_{z=0} \\
\lambda_{77}\partial_2\theta_{77}|_{z=0} + \lambda_{64}\partial_2\theta_{64}|_{z=0} + \lambda_{51}\partial_2\theta_{51}|_{z=0} = \lambda_{46}\partial_2\theta_{46}|_{z=0}\\
\lambda_{77}\partial_3\theta_{77}|_{z=0} + \lambda_{64}\partial_3\theta_{64}|_{z=0} + \lambda_{51}\partial_3\theta_{51}|_{z=0} = \lambda_{46}\partial_3\theta_{46}|_{z=0} \\
\end{array}
\right. \quad \quad \forall \tau \in \mathcal{H}_3,
\ee
where $\partial_k\theta_{ij}|_{z=0} :=\frac{\partial}{\partial z_k}\theta_{ij}|_{z=0} $, with $k=1,2,3$.
The solution of (\ref{system}) can be determined up to a constant:
\bes
\begin{array}{cccc}
\lambda_{77}=D(46, 64, 51),
 & \lambda_{64} =D(77, 64, 46), 
&\lambda_{51} = D(77, 46, 51),
&\lambda_{46} =D(77, 64, 51). 
\end{array}
\ees\\
Here, as in Section \ref{bitangmatrix}, $D(l,m,n):=\rm{det}\frac{\partial (\theta_l, \theta_m, \theta_n)}{\partial z_1\partial z_2 \partial z_3}$.
\ni By repeating this procedure on each row we get the matrix:\\
\bes
M=
\begin{pmatrix}
0 & D(46,64,51)b_{77} & D(77,46,51)b_{64} & D(77,64,46)b_{51} & D(77,64,51)b_{46} \\
D(31,13,26)b_{77} & 0 & D(77,31,26)b_{13} & D(77,13,31)b_{26} & D(77,13,26)b_{31} \\
D(22,13,35)b_{64} & D(64,22,35)b_{13} & 0 & D(64,13,22)b_{35} & D(64,13,35)b_{22} \\
D(17,26,35)b_{51} & D(51,17,35)b_{26} & D(51,26,17)b_{35} & 0 & D(51,26,35)b_{17} \\
D(17,31,22) b_{46} & D(46,17,22)b_{31} & D(46,31,17)b_{22} & D(46,31,22)b_{17} & 0 \\
\end{pmatrix}.
\ees
\\
 Although this matrix is not symmetric, it can be turned into a symmetric one by multiplying it on the left by a suitable diagonal matrix (note that this operation does not change the rank). If we choose the matrix $D_1$:
\bes
D_1:=
\rm{diag}
\left(
1,
\frac{D(46,64,51)}{D(31,13,26)},
\frac{D(77,46,51)}{D(22,13,35)},
\frac{D(77,64,46)}{D(17,26,35)},
\frac{D(77,64,51)}{D(17,31,22)}
\right),
\ees
we set $S'_1:=D_1M$ and  we get:\\
\\
\resizebox*{1\textwidth}{!}{
$
S'_1=
\begin{pmatrix}
0 & D(46,64,51)b_{77} & D(77,46,51)b_{64} & D(77,64,46)b_{51} & D(77,64,51)b_{46} \\
D(46,64,51)b_{77} & 0 & \frac{D(46,64,51)D(77,31,26)}{D(31,13,26)}b_{13} & \frac{D(46,64,51)D(77,13,31)}{D(31,13,26)}b_{26} & \frac{D(46,64,51)D(77,13,26)}{D(31,13,26)}b_{31} \\
D(77,46,51)b_{64} &\frac{D(77,46,51)D(64,22,35)}{D(22,13,35)} b_{13} & 0 & \frac{D(77,46,51)D(64,13,22)}{D(22,13,35)}b_{35} & \frac{D(77,46,51)D(64,13,35)}{D(22,13,35)}b_{22} \\
D(77,64,46)b_{51} & \frac{D(77,64,46)D(51,17,35)}{D(17,26,35)}b_{26} & \frac{D(77,64,46)D(51,26,17)}{D(17,26,35)}b_{35} & 0 & \frac{D(77,64,46)D(51,26,35)}{D(17,26,35)}b_{17} \\
D(77,64,51) b_{46} & \frac{D(77,64,51)D(46,17,22)}{D(17,31,22)}b_{31} & \frac{D(77,64,51)D(46,31,17)}{D(17,31,22)}b_{22} & \frac{D(77,64,51)D(46,31,22)}{D(17,31,22)}b_{17} & 0 \\
\end{pmatrix}.
$
}
\\
\\
\\
Thanks to the relations among the determinants induced by Jacobi's derivative formula \cite{ig80}, the matrix $S'_1$ is easily seen to be symmetric. Note that a different diagonal matrix $D_i$ can be chosen for this operation in such a way that the matrix $S'_i:=D_iM$ and the matrix $M$ have the same entries on the $i$-th row. A straightforward computation proves that $D_iD_1^{-1}=c_iId$ with a suitable $c_i$, hence $S'_i=c_iS'_1$.\\
\\
We can get a more convenient form for $S'_1$ acting by congruence with   the diagonal matrix:
\bes
T_1:=\rm{diag}
\left(
1,
\frac{D(31,13,26)}{D(46,64,51)D(77,31,26)},
\frac{D(22,13,35)}{D(77,46,51)},
1,
1
\right).
\ees
\ni Then we have:\\
\\
\resizebox*{1\textwidth}{!}{
$
S_1:=T_1S'_1T_1=
\begin{pmatrix}
0 & \frac{D(31,13,26)}{D(77, 31,26)}b_{77} & D(22,13,35)b_{64} & D(77,64,46)b_{51} & D(77,64,51)b_{46} \\
\frac{D(31,13,26)}{D(77, 31,26)}b_{77} & 0 & \frac{D(22,13,35)}{D(77, 46, 51)}b_{13} & \frac{D(77, 13, 31)}{D(77, 31, 26)}b_{26} & \frac{D(77,13,26)}{D(77, 31, 26)}b_{31} \\
D(22,13,35)b_{64} &\frac{D(22,13,35)}{D(77, 46, 51)} b_{13} & 0 & D(64, 13, 22)b_{35} & D(64, 13, 35)b_{22} \\
D(77,64,46)b_{51} & \frac{D(77, 13, 31)}{D(77, 31, 26)}b_{26} & D(64, 13, 22)b_{35} & 0 & \frac{D(77,64,46)D(51,26,35)}{D(17,26,35)}b_{17} \\
D(77,64,51) b_{46} &  \frac{D(77,13,26)}{D(77, 31, 26)}b_{31} & D(64, 13, 35)b_{22} & \frac{D(77,64,46)D(51,26,35)}{D(17,26,35)}b_{17} & 0 \\
\end{pmatrix}.
$
}\\
\\
\\
Likewise, the whole procedure can be repeated for the submatrices of $\mathcal{M}$ obtained by replacing the fifth column and row respectively with the sixth, the seventh and the eighth. Then we get the following symmetric matrices of rank $4$:
\\
\\
\resizebox*{1\textwidth}{!}{
$
S_2:=
\begin{pmatrix}
0 &\frac{D(54,13,26)}{D(77, 54,26)}b_{77} & D(47,13,35)b_{64} & D(77,64,23)b_{51} & D(77,64,51)b_{23} \\
\frac{D(54,13,26)}{D(77, 54,26)}b_{77} & 0 & \frac{D(47,13,35)}{D(77, 23,51)}b_{13} & \frac{D(77,13,54)}{D(77,54,26)}b_{26} & \frac{D(77,13,26)}{D(77,54, 26)}b_{54} \\
 D(47,13,35)b_{64} & \frac{D(47,13,35)}{D(77, 23,51)} b_{13} & 0 & D(64,13, 47)b_{35} &  D(64,13, 35)b_{47} \\
D(77,64,23)b_{51} & \frac{D(77,13,54)}{D(54,13, 26)}b_{26} & D(64,13, 47)b_{35} & 0 & \frac{D(77,64,23)D(51,26,35)}{D(72,26,35)}b_{72} \\
D(77,64,51)b_{23} & \frac{D(77,13,26)}{D(54,13, 26)}b_{54} & D(64,13, 35)b_{47} & \frac{D(77,64,23)D(51,26,35)}{D(72,26,35)}b_{72} & 0 \\
\end{pmatrix},
$
}
\\
\\
\resizebox*{1\textwidth}{!}{
$
S_3:=
\begin{pmatrix}
0 & \frac{D(62,13,26)}{D(77,62,26)}b_{77} & D(71,13,35)b_{64} & D(77,64,15)b_{51} & D(77,64,51)b_{15} \\
\frac{D(62,13,26)}{D(77,62,26)}b_{77} & 0 & \frac{D(71,13,35)}{D(77,15,51)}b_{13} & \frac{D(77,13,62)}{D(77,62,26)}b_{26} & \frac{D(77,13,26)}{D(77,62,26)}b_{62} \\
D(71,13,35)b_{64} & \frac{D(71,13,35)}{D(77,15,51)}b_{13} & 0 & D(64,13,71)b_{35} & D(64,13,35)b_{71} \\
D(77,64,15b_{51} & \frac{D(77,13,62)}{D(77,62,26)}b_{26} & D(64,13,71)b_{35} & 0 & \frac{D(77,64,15)D(51,26,35)}{D(44,26,35)}b_{44} \\
D(77,64,51) b_{15} & \frac{D(77,13,26)}{D(77,62,26)}b_{62} & D(64,13,35)b_{71} & \frac{D(77,64,15)D(51,26,35)}{D(44,26,35)}b_{44} & 0 \\
\end{pmatrix},
$
}
\\
\\
\resizebox*{1\textwidth}{!}{
$
S_4:=
\begin{pmatrix}
0 & \frac{D(45,13,26)}{D(77, 45,26)}b_{77} & D(56,13,35)b_{64} & D(77,64,32)b_{51} & D(77,64,51)b_{32} \\
\frac{D(45,13,26)}{D(77, 45,26)}b_{77} & 0 & \frac{D(56,13,35)}{D(77, 32, 51)}b_{13} & \frac{D(77, 13, 45)}{D(77, 45, 26)}b_{26} & \frac{D(77,13,26)}{D(77, 45, 26)}b_{45} \\
D(56,13,35)b_{64} & \frac{D(56,13,35)}{D(77, 32, 51)}b_{13} & 0 & D(64, 13, 56)b_{35} & D(64, 13, 35)b_{56} \\
 D(77,64,32)b_{51} & \frac{D(77, 13, 45)}{D(77, 45, 26)}b_{26} & D(64, 13, 56)b_{35} & 0 & \frac{D(77,64,32)D(51,26,35)}{D(63,26,35)}b_{63} \\
D(77,64,51) b_{46} &  \frac{D(77,13,26)}{D(77, 45, 26)}b_{45} & D(64, 13, 35)b_{56} & \frac{D(77,64,32)D(51,26,35)}{D(63,26,35)}b_{63} & 0 \\
\end{pmatrix}.
$
}\\
\\
\\
We can  act by congruence on  $S_2$, $S_3$ and $S_4$ using the diagonal matrices: 
\bes
\begin{array}{ccc}
N_2:=
\rm{diag}
\left(
\sqrt{A},
\sqrt{A}\frac{D(77,23,51)}{D(77,46,51)},
\frac{1}{\sqrt{A}}\frac{D(22,13,35)}{D(47,13,35)},
\frac{1}{\sqrt{A}}\frac{D(77,64,46)}{D(77,64,23)},
\frac{1}{\sqrt{A}}
\right), & & A := \frac{D(77,46,51) D(31, 13, 26) D(77,54,26)}{D(77,31,26) D(77, 23, 51) D(54,13,26)},\\
& & \\
N_3:=
\rm{diag}
\left(
\sqrt{B},
\sqrt{B}\frac{D(77,15,51)}{D(77,46,51)},
\frac{1}{\sqrt{B}}\frac{D(22,13,35)}{D(71,13,35)},
\frac{1}{\sqrt{B}}\frac{D(77,64,46)}{D(77,64,15)},
\frac{1}{\sqrt{B}}
\right), & & B := \frac{D(77,46,51) D(77, 62, 26) D(31,13,26)}{D(62,13,26) D(77, 15, 51) D(77,31,26)},\\
& & \\
N_4:=
\rm{diag}
\left(
\sqrt{C},
\sqrt{C}\frac{D(77,32,51)}{D(77,46,51)},
\frac{1}{\sqrt{C}}\frac{D(22,13,35)}{D(56,13,35)},
\frac{1}{\sqrt{C}}\frac{D(77,64,46)}{D(77,64,32)},
\frac{1}{\sqrt{C}} \right), & & C = \frac{D(77,46,51) D(77, 45, 26) D(31,13,26)}{D(45,13,26) D(77, 32, 51) D(77,31,26)}.
\\
\end{array}
\ees

\ni Then $S_1$, $N_2S_2N_2$, $N_3S_3N_3$ and $N_4S_4N_4$ have the same entries on the common rows and columns; hence, we have  the following $8 \times 8$ symmetric matrix: \\
\\
\resizebox*{1\textwidth}{!}{
$
\begin{pmatrix}
0 & \frac{D(31,13,26)}{D(77,31,26)}b_{77} & D(22,13,35)b_{64} & D(77,64,46)b_{51} & D(77,64,51)b_{46} & D(77,64,51)b_{23} & D(77,64,51)b_{15} & D(77,64,51)b_{32} \\
& & & & \\
* & 0 & \frac{D(22,13,35)}{D(77,46,51)}b_{13} & \frac{D(77,13,31)}{D(77,31,26)}b_{26} & \frac{D(77,13,26)}{D(77,31,26)}b_{31} & \frac{1}{A}\frac{D(31,13,26)D(77,13,26)}{D(54,13,26)D(77,31,26)}b_{54} & \frac{1}{B}\frac{D(31,13,26)D(77,13,26)}{D(62,13,26)D(77,31,26)}b_{62} & \frac{1}{C}\frac{D(31,13,26)D(77,13,26)}{D(45,13,26)D(77,31,26)}b_{45} \\
& & & & \\
*&* & 0 & D(64,13,22)b_{35} & D(64,13,35)b_{22} & \frac{1}{A}\frac{D(22,13,35)D(64,13,35)}{D(47,13,35)}b_{47} & \frac{1}{B}\frac{D(22,13,35)D(64,13,35)}{D(71,13,35)}b_{71} & \frac{1}{C}\frac{D(22,13,35)D(64,13,35)}{D(56,13,35)}b_{56} \\
& & & & \\
*& * &* & 0 & \frac{D(77,64,46)D(51,26,35)}{D(17,26,35)}b_{17} & \frac{1}{A}\frac{D(77,64,46)D(51,26,35)}{D(72,26,35)}b_{72} & \frac{1}{B}\frac{D(77,64,46)D(51,26,35)}{D(44,26,35)}b_{44} & \frac{1}{C}\frac{D(77,64,46)D(51,26,35)}{D(63,26,35)}b_{63} \\
& & & & \\
* & *& * & * & 0 & X_{65}b_{65} & X_{53}b_{53} & X_{74}b_{74} \\
& & & & \\
* & * & * & * & * & 0 & X_{36}b_{36} & X_{11}b_{11} \\
& & & & \\
* & * &* & * & * & * & 0 & X_{27}b_{27} \\
& & & & \\
* & * & * & * & * & * & *& 0 \\
\end{pmatrix},
$
}\\
\\
\\
\ni where the $X_{ij}$ are to be determined in such a way that the rank of the matrix is equal to $4$. For this purpose we note that we have determined the minors $S_1$, $N_2S_2N_2$, $N_3S_3N_3$ and $N_4S_4N_4$ by demanding precise relations among the eight vector columns $V_i$ of the $8 \times 8$ matrix $\mathcal{M}$. If we set:\\
\bes
\begin{array}{ccccc}
& c_2:=\frac{D(46, 64, 51) D(77, 31, 26)}{D(31, 13, 26)}; & c_3:=\frac{D(77, 46, 51)}{D(22, 13, 35)}; & & \\
\\
d_1:=\frac{1}{\sqrt{A}}; \ & d_2:=\frac{1}{\sqrt{A}}\frac{D(23, 64, 51) D(77, 54, 26) D(77, 46, 51)}{D(54, 13, 26)D(77, 23, 51)}; & d_3:=\sqrt{A}\frac{D(22, 13, 35)}{D(77, 23, 51)}; & d_4:=\sqrt{A}\frac{D(77, 64, 23)}{D(77, 64, 46)}; &  d_6:=\sqrt{A};\\
\\
e_1:=\frac{1}{\sqrt{B}}; \ & e_2:=\frac{1}{\sqrt{B}}\frac{D(15, 64, 51) D(77, 62, 26) D(77, 46, 51)}{D(62, 13, 26)D(77, 15, 51)}; & e_3:=\sqrt{B}\frac{D(77, 15, 51)}{D(22, 13, 35)}; & e_4:=\sqrt{B}\frac{D(77, 64,15)}{D(77, 64, 46)}; &  e_7:=\sqrt{B};\\
\\
f_1:=\frac{1}{\sqrt{C}}; \ & f_2:=\frac{1}{\sqrt{C}}\frac{D(32, 64, 51) D(77, 45, 26) D(77, 46, 51)}{D(45, 13, 26)D(77, 32, 51)}; & f_3:=\sqrt{C}\frac{D(77, 32, 51)}{D(22, 13, 35)}; & f_4:=\sqrt{C}\frac{D(77, 64,32)}{D(77, 64, 46)}; &  f_8:=\sqrt{C};\\
\end{array}
\ees
\\
\\
\ni then the following relations hold:
\begin{align*}
V_1+c_2V_2+c_3V_3+V_4-V_5&=0,\\
d_1V_1+d_2V_2+d_3V_3+d_4V_4-d_6V_6&=0,\\
e_1V_1+e_2V_2+e_3V_3+e_4V_4-e_7V_7&=0,\\
f_1V_1+f_2V_2+f_3V_3+f_4V_4-f_8V_8&=0,
\end{align*}
\ni each respectively on the rows and the columns of the corresponding $5\times5$ minor.\\
\ni In particular, the following relation holds on the first four rows:
\bes
\label{r3}
\left (c_2 -\frac{d_2}{d_1}\right )V_2+\\
+\left (c_3 -\frac{d_3}{d_1}\right )V_3+\left (1- \frac{d_4}{d_1}\right )V_4-V_5+\frac{d_6}{d_1}V_6=0.
\ees
\\Hence, it can be used to compute $X_{65}$, by demanding it on the fifth row:\\
$$
X_{65}^{(5)}= \frac{1}{A}\left (A\cdot D(77,23,51) - D(77,46,51) \right )\frac{D(22,31,17)D(64,13,35)}{D(65,31,17)D(22,13,35)},
\\
$$
Otherwise we can compute $X_{65}$ by demanding the relation on the sixth row, we get:
\bes
X_{65}^{(6)}= \left (\frac{1}{A} - \frac{D(77,64,23)}{D(77,64,46)} \right )\frac{D(77,64,46)D(51,26,35)D(72,54,47)}{D(72,26,35)D(65,54,47)}.
\ees
\ni Writing the two formulas in terms of theta constants, we have
\begin{align*}
X_{65}^{(5)}&=\pm\theta_{14}\theta_{33}\left ( \frac{\theta_{00}\theta_{42}\theta_{57}\theta_{61}\theta_{70}}{\theta_{52}\theta_{75}}\right )\left ( \frac{\theta_{02}\theta_{03} \theta_{24}\theta_{25} }{\theta_{41}\theta_{40}\theta_{66}\theta_{67}} - 1\right ),\\
X_{65}^{(6)}&=\pm\theta_{06}\theta_{21}\left ( \frac{\theta_{00}\theta_{42}\theta_{57}\theta_{61}\theta_{70}}{\theta_{40}\theta_{67}} \right )\left (1 - \frac{\theta_{03} \theta_{10}\theta_{24}\theta_{37}}{\theta_{52}\theta_{41}\theta_{66}\theta_{75}}\right).
\end{align*}
By virtue of the Riemann relations in genus $3$:
\bes
\begin{array}{ccc}
\theta_{52}\theta_{75}\theta_{41 }\theta_{66}-\theta_{03}\theta_{10}\theta_{24 }\theta_{37}=
\theta_{14}\theta_{07}\theta_{33 }\theta_{20}, & \quad \quad & \theta_{40}\theta_{67}\theta_{41 }\theta_{66}-\theta_{03}\theta_{02}\theta_{24 }\theta_{25}=
\theta_{06}\theta_{21}\theta_{07 }\theta_{20}
\end{array}
\ees

\ni the two expressions for $X_{65}$ turn out to be equal:\\
$$
X_{65}^{(5)}=\theta_{14}\theta_{33}\left ( \frac{\theta_{00}\theta_{42}\theta_{57}\theta_{61}\theta_{70}}{\theta_{52}\theta_{75}}\right ) \frac{\theta_{06}\theta_{21} \theta_{07}\theta_{20} }{\theta_{41}\theta_{40}\theta_{66}\theta_{67}}  = \theta_{06}\theta_{21}\left ( \frac{\theta_{00}\theta_{42}\theta_{57}\theta_{61}\theta_{70}}{\theta_{40}\theta_{67}} \right ) \frac{\theta_{14} \theta_{07}\theta_{33}\theta_{20}}{\theta_{52}\theta_{41}\theta_{66}\theta_{75}} = X_{65}^{(6)}.
$$
\\
\ni Likewise we get:

\bes
\begin{array}{l}
X_{53}= \frac{1}{B}\left (B D(77,15,51) - D(77,46,51) \right )\frac{D(22,31,17)D(64,13,35)}{D(53,31,17)D(22,13,35)};\\
X_{74}= \frac{1}{C}\left (1- C\frac{D(77, 64, 32)}{D(77, 64, 46)} \right )\frac{D(77,64,51)D(46,31,22)}{D(74,31,22)};\\
X_{36}= \frac{1}{B}\left ( 1-\frac{D(54, 13, 26) D(77, 23, 51)D(15,64,51)D(77,62,26)}{D(23,64,51)D(77,54,26))D(77,15,51)D(62,13,26)} \right )\frac{D(77,64,51)D(23,47,72)}{D(36,47,72)};\\
X_{11}=\left(\frac{1}{C}-\frac{D(77,64,32)}{AD(77,64,23)} \right)\frac{D(23,54,47)D(77,64,51)}{D(11,54,47)};\\
X_{27}=\left(\frac{1}{C}-\frac{D(77,64,32)}{BD(77,64,15)} \right)\frac{D(15,62,71)D(77,64,51)}{D(27,62,71)}.
\end{array}
\ees

\ni Hence, each entry of the $8\times8$ symmetric matrix with rank equal to $4$ is  uniquely   determined, up to  congruences   by diagonal  matrices. In particular, we will get a suitable form for the matrix we have determined, by multiplying it on both sides by the diagonal matrix $\rm{diag}
\left(
1,
D(77,31,26),
1,
1,
1,
1,
1,
1 
\right)$:
\\
\\
\resizebox*{1\textwidth}{!}{
$
\begin{pmatrix}
0 & D(31,13,26)b_{77} & D(22,13,35)b_{64} & D(77,64,46)b_{51} & D(77,64,51)b_{46} & D(77,64,51)b_{23} & D(77,64,51)b_{15} & D(77,64,51)b_{32} \\
& & & & \\
* & 0 & \frac{D(22,13,35)D(77,31,26)}{D(77,46,51)}b_{13} & D(77,13,31)b_{26} & D(77,13,26)b_{31} & \frac{1}{A}\frac{D(31,13,26)D(77,13,26)}{D(54,13,26)}b_{54} & \frac{1}{B}\frac{D(31,13,26)D(77,13,26)}{D(62,13,26)}b_{62} & \frac{1}{C}\frac{D(31,13,26)D(77,13,26)}{D(45,13,26)}b_{45} \\
& & & & \\
*&* & 0 & D(64,13,22)b_{35} & D(64,13,35)b_{22} & \frac{1}{A}\frac{D(22,13,35)D(64,13,35)}{D(47,13,35)}b_{47} & \frac{1}{B}\frac{D(22,13,35)D(64,13,35)}{D(71,13,35)}b_{71} & \frac{1}{C}\frac{D(22,13,35)D(64,13,35)}{D(56,13,35)}b_{56} \\
& & & & \\
*& * &* & 0 & \frac{D(77,64,46)D(51,26,35)}{D(17,26,35)}b_{17} & \frac{1}{A}\frac{D(77,64,46)D(51,26,35)}{D(72,26,35)}b_{72} & \frac{1}{B}\frac{D(77,64,46)D(51,26,35)}{D(44,26,35)}b_{44} & \frac{1}{C}\frac{D(77,64,46)D(51,26,35)}{D(63,26,35)}b_{63} \\
& & & & \\
* & *& * & * & 0 & \frac{1}{A}\left (A\cdot D(77,23,51) - D(77,46,51) \right )\frac{D(22,31,17)D(64,13,35)}{D(65,31,17)D(22,13,35)}b_{65} &\frac{1}{B}\left (B D(77,15,51) - D(77,46,51) \right )\frac{D(22,31,17)D(64,13,35)}{D(53,31,17)D(22,13,35)}b_{53} &  \frac{1}{C}\left (1- C\frac{D(77, 64, 32)}{D(77, 64, 46)} \right )\frac{D(77,64,51)D(46,31,22)}{D(74,31,22)}b_{74} \\
& & & & \\
* & * & * & * & * & 0 &\frac{1}{B}\left ( 1-\frac{D(54, 13, 26) D(77, 23, 51)D(15,64,51)D(77,62,26)}{D(23,64,51)D(77,54,26))D(77,15,51)D(62,13,26)} \right )\frac{D(77,64,51)D(23,47,72)}{D(36,47,72)}b_{36} &\left(\frac{1}{C}-\frac{D(77,64,32)}{AD(77,64,23)} \right)\frac{D(23,54,47)D(77,64,51)}{D(11,54,47)}b_{11} \\
& & & & \\
* & * &* & * & * & * & 0 & \left(\frac{1}{C}-\frac{D(77,64,32)}{BD(77,64,15)} \right)\frac{D(15,62,71)D(77,64,51)}{D(27,62,71)}b_{27} \\
& & & & \\
* & * & * & * & * & * & *& 0 \\
\end{pmatrix}.
$
}
\\
\\
\ni Using the expression of the Jacobian determinant in terms of theta constants and the Riemann relations we get the matrix $L(\tau, z)$:\\
\\
 \resizebox*{1\textwidth}{!}{
$
\label{Z3}
L(\tau, z)=\begin{pmatrix}
0 & D(31,13,26)b_{77} & D(22,13,35)b_{64} & D(77,64,46)b_{51} & D(77,64,51)b_{46} & D(77,64,51)b_{23} & D(77,64,51)b_{15} & D(77,64,51)b_{32} \\
& & & & \\
& & & & \\
* & 0 & \pm \frac{\theta_{60}}{\theta_{04}} D(77, 31, 26)b_{13} &D(77,13,31)b_{26} & D(77,13,26)b_{31} & \pm \left (\frac{\theta_{07} \theta_{10}\theta_{25}\theta_{61}\theta_{73}}{\theta_{04}\theta_{40}\theta_{52}\theta_{67}\theta_{75}}\right )D(77,23,32)  b_{54} & \pm \left (\frac{\theta_{07} \theta_{10}\theta_{25}\theta_{57}\theta_{73}}{\theta_{04}\theta_{43}\theta_{52}\theta_{67}\theta_{76}}  \right ) D(77,15,32) b_{62} & \pm \left ( \frac{\theta_{01}  \theta_{16} \theta_{34}\theta_{70}\theta_{73}}{\theta_{04}\theta_{40}\theta_{43}\theta_{75}\theta_{76}}\right )D(77, 23, 32)b_{45} \\
& & & & \\
*&* & 0 & D(64,13,22)b_{35} &  D(64,13,35)b_{22} & \pm \left (\frac{\theta_{03} \theta_{14} \theta_{25}\theta_{60}\theta_{61}}{\theta_{04}\theta_{40}\theta_{41}\theta_{66}\theta_{67}} \right ) D(64,23,32)b_{47} & \pm \left ( \frac{\theta_{03} \theta_{14}\theta_{25}\theta_{57}\theta_{60}}{\theta_{04}\theta_{41}\theta_{50}\theta_{67}\theta_{76}} \right ) D(64, 15, 32) b_{71} & \pm \left (\frac{ \theta_{05} \theta_{12} \theta_{34}\theta_{60}\theta_{70}}{\theta_{04}\theta_{40}\theta_{50}\theta_{66}\theta_{76}}\right ) D(64, 23, 32)b_{56} \\
& & & & \\
*& * &* & 0 &  \pm \left ( \frac{\theta_{04}}{\theta_{55}} \right ) D(51, 26, 35)b_{17} & \pm \left ( \frac{\theta_{03} \theta_{10} \theta_{21}\theta_{61}}{\theta_{41}\theta_{52}\theta_{66}\theta_{75}}\right ) D(51,23,32)b_{72} & \pm \left (\frac{\theta_{03} \theta_{10} \theta_{21}\theta_{57}}{\theta_{41}\theta_{43}\theta_{50}\theta_{52}} \right )D(51, 15, 32)b_{44} & \pm \left( \frac{ \theta_{01}  \theta_{12} \theta_{30}\theta_{70}}{\theta_{43}\theta_{50}\theta_{66}\theta_{75}}\right )D(51,23,32)b_{63} \\
& & & & \\
* & *& * & * & 0 & \pm \left (\frac{\theta_{06}\theta_{07}\theta_{14} \theta_{20}\theta_{21} \theta_{33}}{\theta_{40}\theta_{41}\theta_{52}\theta_{66}\theta_{67}\theta_{75}} \right ) D(22, 31, 17) b_{65} &\pm \left (\frac{\theta_{05}\theta_{07}\theta_{14}\theta_{16}\theta_{21}\theta_{30}}{\theta_{41}\theta_{43}\theta_{50}\theta_{52}\theta_{67}\theta_{76}} \right ) D(22, 31, 17) b_{53} &  \pm    \left ( \frac{ \theta_{05}\theta_{16} \theta_{20}\theta_{33}\theta_{52} }{\theta_{21}\theta_{40}\theta_{50}\theta_{66}\theta_{76}}\right )D(46,31,22)b_{74} \\
& & & & \\
* & * & * & * & * & 0 & \pm \left (  \frac{\theta_{03}\theta_{07}\theta_{10}
\theta_{14}\theta_{21} \theta_{25} \theta_{55}\theta_{60} \theta_{73}}{\theta_{40}\theta_{41}\theta_{43}\theta_{50}\theta_{52}\theta_{66}\theta_{67}
\theta_{75}\theta_{76} }\right)D(22, 31, 17)b_{36}  & \pm \left (\frac{\ \theta_{06}\theta_{24}\theta_{33} \theta_{42}\theta_{55}\theta_{57}\theta_{60} \theta_{61}\theta_{70} }{\theta_{40}\theta_{41}\theta_{43}\theta_{50}\theta_{52}\theta_{66}\theta_{67}\theta_{75}\theta_{76}}   \right ) D(77, 23, 32)b_{11} \\
& & & & \\
* & * &* & * & * & * & 0 & \pm \left (\frac{\theta_{01}\theta_{05}\theta_{16}\theta_{12}\theta_{30}\theta_{34} \theta_{55} \theta_{60}\theta_{73} }{\theta_{40}\theta_{41}\theta_{43}\theta_{50}\theta_{52}\theta_{66}\theta_{67}\theta_{75}\theta_{76}}\right ) D(22, 31, 17)b_{27}  \\
& & & & \\
* & * & * & * & * & * & *& 0 \\
\end{pmatrix}.
$
}
\\
\\
\\
\ni \begin{remark} Note that each coefficient can be written as a product of at most $8$ determinants over $7$ determinants, although there seems not to be any canonical choice for such an expression.\\
\ni Take, for instance, the coefficient $X_{65}$; the triples of even characteristics $\{(06), (07), (14)\}$ and $\{(40), (41), (52)\}$ extend to azygetic $5$-tuples by means of the same pair $\{(55), (70)\}$, and the triples $\{(20), (21), (33)\}$ and $\{(66), (67), (75)\}$ extend to azygetic $5$-tuples by means of the pair $\{(34), (70)\}$. Therefore, we can write:\\
\bes
X_{65} =\pm \frac{\theta_{55}\theta_{70}}{\theta_{55}\theta_{70}} \left ( \frac{\theta_{06}\theta_{07}\theta_{14} }{\theta_{40}\theta_{41}\theta_{52}}\right )\cdot \frac{\theta_{34}\theta_{70}}{\theta_{34}\theta_{70}}\left ( \frac{ \theta_{20}\theta_{21} \theta_{33}}{\theta_{66}\theta_{67}\theta_{75}}\right ) D(22, 31, 17) = \pm \frac{D(11, 53, 72)D(11,13,74)}{D(11,15,72)D(11,13,32)}D(22, 31, 17).
\ees
In a similar way, we can get such an expression for $X_{13}$ and  $X_{53}$
\begin{align*}
X_{13} &=\pm \frac{\theta_{60}}{\theta_{04} }   D(77, 31, 26) = \frac{D(22, 13, 35)} {D(77, 46, 51)} D(77, 31, 26), \\
X_{53} &=\pm \frac{\theta_{37}\theta_{61}}{\theta_{37}\theta_{61}} \left ( \frac{\theta_{05}\theta_{07}\theta_{21}}{\theta_{41}\theta_{43}\theta_{67}}\right )\cdot \frac{\theta_{37}\theta_{70}}{\theta_{37}\theta_{70}}\left ( \frac{\theta_{14}\theta_{16}\theta_{30}}{\theta_{50}\theta_{52}\theta_{76}}\right )D(22, 31, 17) = \pm \frac{D(26, 36, 65)D(26,27,74)}{D(23,26,36)D(26,27,32)}D(22, 31, 17).
\end{align*}
\ni and so on for each entry of the matrix.
\end{remark}

We  can  summarize the previous discussion in the following  way.
 \begin{theorem}\label{main}  Let   $\tau$  be the period matrix of  the jacobian of a smooth plane quartic. When an even  characteristic and a corresponding Aronhold set of characteristics (i.e. a level  2 structure) are fixed then, up to congruences, a
 unique matrix $L(\tau, z)$ of  rank four is determined 
 in such a way that its entries are proportional to the  $28$ bitangents. The equation of the corresponding plane quartic is obtained   taking the determinant of any minor of degree  four of the  above matrix.
\end{theorem}

Hence  for determining the  matrix $A(z)$  we can   consider the minor obtained  taking the first 4 rows and columns of the  matrix  $L(\tau, z)$  divided  for a suitable  jacobian determinant, so  we  will get  modular functions as coefficients, as stated in the following corollary.
 \begin{corollary}  Let   $\tau$  be the period matrix of   the jacobian of a smooth plane quartic, then the  matrix $A(z)$ is  congruent to the following matrix:
\bes
\label{Z4}
Q(\tau, z)=\begin{pmatrix}
0 & \frac {D(31,13,26)}{D(77, 31, 26)}b_{77} & \frac {D(22,13,35)}{D(77, 31, 26)}b_{64} & \frac{D(77,64,46)}{D(77, 31, 26)}b_{51} \\
* & 0 &  \frac{D(22, 13, 35)} {D(77, 46, 51)}  b_{13} &\frac{D(77,13,31)}{D(77, 31, 26)}b_{26}  \\
*&* & 0 &\frac{ D(64,13,22)}{D(77, 31, 26)}b_{35} \\
*& * &* & 0 \\
\end{pmatrix}.
\ees
\\
\ni
Moreover  
$${\rm det}\, Q(\tau, z)=0,$$
 is an equation for the  plane quartic. 
\end{corollary}
 \ni A very similar equation for the plane quartic has been already obtained in \cite{gu}.\medskip
 This is a classical result that  goes  back  to Riemann; we refer to \cite{Do} for details.


\end{document}